\documentclass[times]{amsart}
\usepackage{amssymb, hyperref, fullpage}

\theoremstyle{plain}
\newtheorem{proposition}[equation]{Proposition}
\newtheorem{remarks}[equation]{Remarks}
\newtheorem{theorem}[equation]{Theorem}
\newtheorem{definition}[equation]{Definition}
\newtheorem{corollary}[equation]{Corollary}
\newtheorem{lemma}[equation]{Lemma}
\newtheorem{remark}[equation]{Remark}
 \newcommand{\esm}{\end{smallmatrix}\right)}
\newcommand{\bpm}{\begin{pmatrix}}
\newcommand{\epm}	{\end{pmatrix}}
\newcommand{\bspm}{\left(\begin{smallmatrix}}
\newcommand{\espm}	{\end{smallmatrix}\right)}
\newcommand{\R}{\mathbb R}
\newcommand{\A}{\mathbb A}
\newcommand{\Z}{\mathbb Z}

\newcommand{\Gm}		{\Gamma}

\newcommand{\fa}		{\mathfrak {a}}
\newcommand{\gm}		{\gamma}

\newcommand{\N}		{\mathbb {N}}
\newcommand{\Q}		{\mathbb {Q}}
\newcommand{\C} {\mathbb C}
\newcommand{\sg}		{\sigma}

\renewcommand{\P}{\mathbb P}
\newcommand{\bs} {\backslash}

\newcommand{\ord}{\operatorname{ord}}

\newcommand{\fh}{\mathfrak{h}}
\newcommand{\on}{\operatorname}
\newcommand{\sgn}{\on{sgn}}

\newcommand{\wt}{\widetilde}

\begin{document}

\title{On multiplicativity of 
Fourier coefficients at cusps other than infinity
}

\author{Joseph Hundley}

\address{
Department of Mathematics,
Mailcode 4408,
Southern Illinois University,
1245 Lincoln Drive,
Carbondale, IL 62901
}

\begin{abstract}
\end{abstract}

\maketitle
\tableofcontents
\section{Introduction}\label{s:Intro}
This paper is a continuation of \cite{GHL}, concerning the  question, which was posed by Harold Stark, 
\vskip 10pt
  \centerline{\it When do the Fourier coefficients at a cusp of a classical newform for $GL(2)$ satisfy multiplicative relations?}
  \vskip 10pt
  In order to discuss this question in more detail, we review some notation from \cite{GHL}.

  Fix $N \in \N,$ a Dirichlet character $\chi \pmod{N}$ and 
  $f: \fh \to \C$ a Maass form for $\Gm_0(N)$ with character $\chi.$  (Here, $\fh$ denotes the complex upper-half plane.) 
 Let $k$ and $\nu$ denote the weight and type of $f,$
 respectively.  This means that 
 $$
  \left(f \big |_k \gamma\right)(z) :=
  \left(\frac{cz+d}{|cz+d|}\right)^{-k} f\left(\frac{az+b}{cz+d}   \right),\qquad \left( \forall \bpm a&b \\c&d \epm \in \Gm_0(N)\right), 
  \qquad \left(\Delta_k f\right)(z) = \nu(1-\nu)f(z),
 $$
 where 
 $$\Delta_k = -y^2\left(\frac{\partial^2}{\partial x^2}+\frac{\partial^2}{\partial y^2}\right)+iky\frac{\partial}{\partial x}$$
	is the weight $k$ Laplacian.

	We shall define two closely related sets of ``Fourier 
	coefficients'' for $f.$  
	
The first set we shall refer to as the ``standard Fourier 
coefficients at a cusp $\fa.$'' This is something of an 
abuse of language, inasmuch as the coefficients in question depend not only on $\fa,$ but also on a choice of $\sg_\fa \in SL(2, \R)$ with $\sg_\fa \infty = \fa,$ such that the stabilizer of $\infty$ in $\sg_\fa^{-1} \Gm_0(N)\sg_\fa$ is generated by $-I_2$ and $\bspm 1&1\\0&1\espm.$  Such a matrix is necessarily of the form 
$\gm_\fa \bspm \sqrt{m_\fa} &0 \\ 0& \sqrt{m_\fa}^{-1}\espm$
for some $m_\fa \in \Z$ (depending only on $\fa,$ and not on $\sg_\fa,$) and $\gm_\fa \in SL(2, \Z).$ 
Once $\sg_\fa$ (or, equivalently, $\gm_\fa$) is chosen, the definition of the coefficient $A(\fa, n)$ is that 
$$
\int_0^1  \left(f\big|_k \sg_\fa \right) (x+iy) e^{-2\pi i (n+\mu_{\fa,\chi}) x}
  \, dx 
  =A(\fa, n) \, W_{\frac{\text{\rm sgn}(n) k}{2}, \; \nu-\frac12} \Big( 4\pi |n+\mu_{\fa, \chi}|\cdot y  \Big),
$$
where $\mu_{\fa, \chi} \in [0,1)\cap \Q$ is the so-called ``cusp parameter'' of $\fa,$ relative to $\chi$ (which depends only on $\fa,$ not on the choice of $\gm_\fa$).

Now, it is known that the standard Fourier coefficients at
the cusp $\infty$ satisfy the multiplicative relations
\begin{equation}\label{MultRels for A(infty, *)}
A\left(\infty, \,\,\prod_p p^{e_p}\right) 
= \prod_p A( \infty, p^e_p),
\end{equation}
provided that $f$ is a Maass-Hecke newform which has been normalized so that $A( \infty, 1) =1.$
Following Stark, we may ask under what conditions such relations hold at other cusps.  

As one might expect, the answer to Stark's question 
depends on how strictly one interprets the phrase ``multiplicative relations.'' Indeed, if one interprets it very loosely, then it follows from theorem 15 of \cite{GHL} that the answer is ``always.'' 

Theorem 15 of \cite {GHL} states that the standard Fourier coefficient $A(\fa, n)$ of a Maass-Hecke newform at {\em any} cusp may be expressed as a product of Fourier coefficients, similar to \eqref{MultRels for A(infty, *)}, but with coefficients at many different cusps appearing on the right-hand-side, and with the terms corresponding not to the factorization of $n,$ but rather to that of the rational number $(n+ \mu_{\fa, \chi})/m_\fa.$ 

Now suppose that we interpret Stark's question more strictly, as asking under what conditions there are relations akin to \eqref{MultRels for A(infty, *)} involving the coefficients at $\fa,$ and only those?  Then we have results of Asai \cite{Asai:1976} and Kojima \cite{Kojima:1979} showing that such relations hold at all cusps, provided that the level $N$ is squarefree or equal to $4$ times an odd prime.  Furthermore, it follows from Asai's proof that for arbitrary $N,$ such relations hold 
whenever $\sg_\fa$ is a Fricke involution.
A generalization of these results was obtained in theorem 35 
of \cite{GHL}.

In order to review this result, it is convenient to introduce a slightly different notation for the Fourier coefficients.  
	To be precise, for $\gm \in SL(2, \Z),\;\alpha \in \Q$ we have
  $$
  \int_0^1 \left(f\big|_k \gm \right) (x+iy) e^{-2\pi i \alpha x}
  \, dx 
  = B_f( \gm, \alpha) 
  W_{\frac{\sgn(\alpha)k}{2}, \nu-\frac{1}{2}}(4\pi|\alpha |y),
  $$
  for some constant $B_f(\gm, \alpha),$ which satisfies  $$
  B_f( \gm, \alpha) = \begin{cases}
A(\fa, m_\fa \alpha - \mu_{\fa, \chi}),& \text{ if } m_\fa \alpha- \mu_{\fa, \chi} \in \Z,\\
0, & \text{ otherwise,}
\end{cases}
  $$
  provided that $\fa = \gm \cdot \infty$ and we take $\gm_\fa=\gm.$  
  
There are two motivations for this alternate notation. 
First, the properties of $B_f(\gm, \alpha),$ depend on $\gm$ and not only on the cusp $\gm \cdot \infty.$  Thus, when one considers a more precise interpretation of 
``multiplicative relations,'' it is necessary to reflect the dependence of $A(\fa, n)$ on the choice of $\gm_\fa.$  
(Cf. \cite{GHL}, remark 40.)
Second,
common sense 
suggests that when $\mu_{\fa, \chi} \ne 0,$ it makes little sense to consider the factorization of the integer $n$ when studying multiplicative properties of the Fourier coefficients.  What one really has is a Fourier expansion 
supported on a subset of $\Q$ which may be identified with $\Z$ using an additive shift.  This additive shift  is unlikely to respect any multiplicative properties. 
Thus, it makes more sense to parametrize  the coefficients by $\alpha \in \frac1{m_\fa}(\Z+\mu_{\fa, \chi}) \subset \Q$ than by $n \in \Z.$     
(This view is 
supported by theorem 15 of \cite{GHL}, and it is
compatible with \eqref{MultRels for A(infty, *)}, and with the results of Asai and Kojima already mentioned, because $\mu_{\fa, \chi}=0$  in all these cases.)  

In order to formulate a strong version Stark's question, we introduce the definition of a multiplicative function.  

\begin{definition}
A function  $h: \Q^\times \to \C$ is multiplicative if there 
exist functions $h_\infty: \{ \pm 1\} \to \C$
and $h_p:\Z \to \C$ for each prime $p$ 
such that $h_p(0) =1$ for all but finitely 
many primes $p,$ and 
\begin{equation}\label{e:multiplicativity}
h\left( \epsilon \prod_p p^{e_p}\right) = 
h_\infty( \epsilon) \cdot \prod_p h_p(e_p),
\end{equation}
for any $\epsilon \in \{ \pm 1\}$ and 
$e_p \in \Z$ with $e_p=0$ for all but finitely many primes $p.$
\end{definition}

The problem which we considered in section 6 of \cite{GHL} and continue with in this paper is:\\  \underline{\it under what conditions on $f$ and $\gm$ is the function $\alpha \mapsto B_f(\gm, \alpha)$ multiplicative?}

One is primarily interested in the case when $f$ is a Maass-Hecke newform.  As explained in \cite{GHL}, 
this condition implies a more technical condition which is more directly applicable to the question at hand, namely,  that the corresponding adelic Whittaker function $W_f(g)$ factors as a product of local Whittaker functions.  
We shall restrict ourselves to those $f$ which satisfy this more technical condition.

If $\fa = \frac{a}{b}$ (with $\gcd(a,b)=1$) is a cusp, define $M= M(\fa) = \gcd(N,b).$  Then $\fa$ can be mapped to $\infty$ by a Fricke involution if and only if $\gcd(M, N/M) =1.$  
For such $\fa,$ it follows from Asai's result that 
$\gm \in SL(2, \Z)$ with $\gm\cdot \infty = \fa$ exists 
such that $B_f(\gm, \cdot)$ is multiplicative.  
It is shown in 
theorem 35 
 of \cite{GHL} that this can be extended to the case when $\gcd(M, N/M)=2$ and $2|| \frac{N}M.$
 (In the case considered by Kojima, this exhausts 
 all cusps not of the first type.)
Now, it is clear that $B_f( \gm_0\gm, \cdot)$ is a scalar multiple of $B_f(\gm , \cdot)$ for any $\gm \in SL(2, \Z),
\; \gm_0 \in \Gm_0(N).$  
It follows from 
this observation, combined with theorem 35 
 of \cite{GHL},  that 
$B_f( \gm, \cdot)$ is multiplicative for all $\gm = \bspm a&b\\ c&d \espm$ such that $N \mid 2cd.$

The purpose of this paper is to prove a partial converse to this result.   

\begin{theorem}\label{MainResult--General}
Suppose that $B_f(\gm, \cdot )$ is multiplicative.  Then 
$N \mid 576 cd.$    If $A_f( \infty, p)\ne 0$ for 
more than half of all primes $p,$ then 
$N \mid 48cd.$
\end{theorem}
(See theorem \ref{suffcond2--final}, remarks \ref{remarks at the end }, and theorem \ref{t:necCond2-simple}.  Sharper but more technical results are given in theorems 
\ref{t:necCond2-simple-sharp} and 
 \ref{Main theorem--precise version at theend}. )
  
\section{Preliminary Results}

\label{section: preliminary results}
\begin{lemma}
 \label{B+M=abcMd}
 Let $N \ge 1$ be an integer and take $M \mid N.$ 
 Let $B$ denote the subgroup of $SL(2, \Z/N\Z)$ consisting of all upper triangular elements.  Then 
 $$\left\{\bpm a& b \\ Mc & d \epm \in SL(2, \Z/N\Z)
\right\}$$
is generated by $\bspm 1&0\\M&1 \espm$ and $B.$
 \end{lemma} 
 \begin{proof}
 Using the isomorphism 
 $SL(2, \Z/N\Z) \cong \prod_{q\mid N} SL(2, \Z/q^{v_q(N)}\Z),$ it suffices to treat the case $N = q^e, M=q^\ell$ with $\ell < e.$ 
 
The identity
\begin{equation}\label{e1}
\bpm a& b \\ q^\ell c & d \epm
= \bpm 1& 0 \\ q^\ell c\bar a & 1 \epm
\bpm a& b \\ 0& \bar a \epm
\end{equation}
suffices, except when $a$ is not be a unit, which can occur only in the case $\ell =0.$
In that case, $c$ is a unit and 
$$
\bpm 0&-1\\ 1& 0 \epm 
=\bpm 1&-1\\ 0& 1 \epm 
\bpm 1&0\\ 1& 1 \epm 
\bpm 1&-1\\ 0& 1 \epm.
$$
is in the subgroup generated by 
  $B$ and $\bspm 1&0\\ q^\ell& 1 \espm=\bspm 1&0\\ 1& 1 \espm.$  Applying 
  \eqref{e1} to $\bspm a'&b'\\c'&d'\espm = \bspm -c&-d\\a&b\espm=
  \bspm 0&-1\\ 1& 0 \espm\bspm a&b\\c&d\espm$
  we get the result in this case as well.
 \end{proof}

Next, 
 it is convenient to  recall the definition of the $p$-adic valuation.
 \begin{definition}
 Let $p$ be a prime.
 The $p$-adic valuation
 of a rational number $\alpha$
 is defined to be the unique integer $k$ such that $\alpha$ may be written as 
  $\frac ab p^k$ with $\gcd(ab,p) =1.$  
 \end{definition}

\begin{proposition}\label{Gm1(M1)}
Let $N \ge 1$ be an integer and $M$ a divisor of $N.$  Let $M_1$ be the integer such that 
$$
v_q(M_1) = \begin{cases}
v_q(M),& q>3, \text{ and/or } v_q(M) =0,\\
v_q(M)+1, & q=3, v_q(M) >0,\\
3, & q=2, v_q(M) =1,\\
v_q(M)+3, & q=2, v_q(M) \ge 2,
\end{cases}
$$
then $\gcd(M_1,N)$ is the smallest integer $k$ such that 
$\Gm_1(k)$ is generated by $\Gm_1(N)$ and the commutator subgroup of 
$\Gm_0(M).$  
\end{proposition}

\begin{proof}
For an integer $k$ and $i=0,1$ let $\bar \Gm_i(k)$ denote the image of $\Gm_i(k)$ in $SL(2, \Z/N\Z),$ and let $C$ denote the commutator subgroup of $\bar \Gm_0(M).$  
Let $\bar \Gm$ denote the subgroup
generated by 
 $C$ and $\bar \Gm_1(N).$
 We show that 
 $\gcd(M_1,N)$ is the smallest integer $k$ such that 
$\bar\Gm_1(k)$ is generated by $\bar\Gm_1(N)$ and the commutator subgroup of 
$\bar\Gm_0(M).$

Once again, using the isomorphism 
$SL(2, \Z/N\Z) \cong \prod_{q\mid N} SL(2, \Z/q^{v_q(N)}\Z),$ it suffices to treat the case $N = q^e, M=q^\ell$ with $\ell < e.$
If $q >5$ then there exists $m \in (\Z/q^e\Z)^\times$ with $m^2 -1 \in (\Z/q^e\Z)^\times.$  
It follows that 
$\bspm 1& 0 \\ (m^2-1)M & 1 \espm$ 
lies in the commutator subgroup $C,$
and thence that 
\ $\bspm 1& 0 \\kM & 1 \espm$ does, for any $k \in \Z.$
The well-known identity 
$$\bpm 1&0\\a& 1\epm
\bpm 1&1\\0& 1\epm 
\bpm 1&0\\ -a(1+a)^{-1} & 1\epm 
\bpm 1&-(a+1)\\0& 1\epm 
=\bpm (a+1)^{-1} &0\\0&(a+1) \epm$$
shows that $\bspm t &0\\0&t^{-1} \espm\in \bar \Gm$
 for any $t \equiv 1 \pmod{M}.$
It then follows easily that $\bar \Gm_1(M)\subset \bar \Gm,$
which completes the proof in this case.

Set $(A,B) := ABA^{-1}B^{-1}.$
Then 
\begin{equation}\label{e:commutator}
\left(\bpm a& b\\ M c & d \epm,\;\;
\bpm e& f\\ M g & h \epm
\right)=
\bpm
*&*\\ d h M (c (e - h) - g (a - d))
  + (b c e g - a c f g + 
    b d g^2 - c^2 f h) M^2
&*\epm.
\end{equation}
This proves that 
\begin{itemize}
\item $C \subset \bar \Gm_0(3^{\ell +1}),$ if $M = 3^\ell,$ with $e>\ell >0,$
( for in this case $a\equiv d$ and $e \equiv h \pmod{3}$),
\item $C \subset \bar \Gm_0(2^{\ell+\min(3, e)}),$ 
if $M=2^\ell$ with $\ell \ge 3$
(for in this case $a \equiv d$ and $e \equiv h \pmod{\min(8, 2^e)}$). 
\end{itemize}
Further, 
$$
\bpm 1&0 \\ 3^{\ell+1} & 1 \epm =
\bpm 1&0\\ (2^2-1) 3^\ell &1\epm
\in C, \qquad \text{whenever } 
M = 3^\ell,  
$$
$$
\bpm 1&0 \\ 2^{\ell+3} & 1 \epm =
\bpm 1&0\\ (3^2-1) 2^\ell &1\epm
\in C, \qquad \text{whenever } 
M = 2^\ell.  
$$
By lemma \ref{B+M=abcMd}, this matrix, together with $\bar \Gm_1(N),$
generates $\bar \Gm_1(M_1),$ for $M=3^\ell, \, \ell >0$ or $M=2^\ell, \ell> 2.$ 

This leaves the cases $M=2,4$ and $M=1$ when $q=2,3.$  

In the case $M=4,$ we known that $\bar\Gm_1(32) \subset \bar \Gm.$  
 To complete the proof,
 one checks that,
 taken mod $32,$ the bottom row of an 
  element of $C$
  is equivalent to one of the following:
 $$
 (0, 1), (0, 9), (0, 17), (0, 25), (16, 5), (16, 13), (16, 21), (16, 
  29),
 $$
 from which it follows that $\bar \Gm_1(16) \not\subset \bar \Gm$ (unless $N \le 16,$ of course). 
 
 In the case $M=2,$ one 
 knows that $\bar \Gm_1(16) \subset \bar\Gm,$ and 
 checks that,
 taken mod $16,$ the bottom row of an 
  element of $C$
  is equivalent to one of the following:
 $$(0, 1), (0, 5), (0, 9), (0, 13), (4, 3), (4, 7), (4, 11), (4, 
  15), (8, 1), (8, 5), (8, 9), (8, 13), (12, 3), (12, 7), (12, 
  11), (12, 15)$$
 (with each possibility being attained).
  It easily follows that $\bar \Gm$ contains $\bar \Gm_1(8)$ but not $\bar \Gm_1(4).$
  
  It remains to prove $\bar \Gm = SL(2, \Z/q^e\Z)$ in the case 
  when $q=2,3,$ and $\ell=0.$ 
 We treat the case $q=3$ in detail.  The 
 case $q=2$ is similar.
  Arguing as above, we find that 
  $\bar \Gm \supset \bar \Gm_1(3).$ 
  Thus, we may as well assume $e=1.$ 
  If $\bar \Gm \ne SL(2, \Z/3\Z)$ then there 
  is a homomorphism $\varphi$ from $SL(2, \Z/3\Z)$ to an abelian group which is trivial on $\bar \Gm_1(3)$ and hence constant on 
  each double coset in $\bar\Gm_1(3) \bs SL(2, \Z/3\Z) / \bar\Gm_1(3).$  There are four
 such cosets: $$\begin{aligned}
 e &:= \left\{\left.\bpm 1&b\\ 0 & 1 \epm \right|  b \in \Z/3\Z \right\},\\
 x&:=  \left\{\left.\bpm -1&b\\ 0 & -1 \epm \right|  b \in \Z/3\Z \right\},\\
 y&:= \left\{ \left.\bpm a&b\\ 1 & d \epm \right|  a,b,d \in \Z/3\Z, \; ad-b=1 \right\},\\
 z&:= \left\{ \left.\bpm a&b\\ -1 & d \epm \right|  a,b,d \in \Z/3\Z, \; ad+b=1 \right\}.
 \end{aligned}
  $$
  Suppose that $\bar \phi$ is a mapping from these four cosets to some abelian group $A,$ such that the function $\phi$ obtained by composing $\bar \phi$ with the 
  natural projection $SL(2, \Z/3\Z) \to \bar\Gm_1(3) \bs SL(2, \Z/3\Z) / \bar\Gm_1(3)$
 is a homomorphism.  
 Let $\bar e, \bar x, \bar y,$ and $\bar z$ denote the respective images of $e,x,y$ and $z$ in $A.$  Clearly, $\bar e$ is the identity and $\bar x^2 = \bar e.$
 Since
 $$
 \bpm 0&1 \\ -1&0 \epm^2 = \bpm 0& -1 \\1&0 \epm^2 = \bpm -1&0\\0&-1 \epm,
 $$  
 it follows that $\bar y^2=\bar z^2 =\bar x.$  On the other hand 
 $$ \bpm 1&0\\ 1& 1\epm^2 =\bpm 1&0\\ -1& 1\epm \implies \bar y^2 = \bar z,$$
 which implies that $\bar z= \bar e,$ whence $\bar x=\bar e.$  Finally $\bar x\bar z = \bar y,$  so it turns out $\phi$ is trivial.   
 This completes the proof of proposition \ref{Gm1(M1)}.
   \end{proof}

\section{Reformulation of the problem, definition of $B_{f,N_1}(\gm, \cdot)$}
\label{section: Reformulation of the problem}

 It will be useful to first isolate the {\it part} of $B_f(\gm, \alpha)$ whose multiplicativity is in doubt.
Using the results and notation of  \cite{GHL}, one has
 \begin{equation}\label{e:baalpha=prod}
 \begin{aligned}
 B_f(\gm, \alpha)&=
  \frac{W_f\left( i_\infty\left(\begin{pmatrix} y_\infty & \\ & 1\end{pmatrix}\right)\cdot g( \gm, \alpha) \right)}{ W_{\frac{\text{sign}(\alpha)k}2, \nu -\frac 12} (4\pi y_\infty)}\\
  & =
  \frac{W_{f, \infty}\begin{pmatrix} \operatorname{sign}(\alpha)y_\infty & \\ & 1\end{pmatrix}}{ W_{\frac{\text{sign}(\alpha)k}2, \nu -\frac 12} (4\pi y_\infty)}
  \cdot \prod_p W_{f,p}\left(\begin{pmatrix} \alpha & \\ & 1\end{pmatrix}\gamma^{-1}\right),
  \end{aligned}
 \end{equation}
  for any  $y_\infty \in(0,\infty)$ such that $ W_{\frac{\text{sign}(\alpha)k}2, \nu -\frac 12} (4\pi y_\infty)\ne 0$
 (cf. \cite{GHL}, equation (37)
 ).
 Here, 
   \begin{equation}\label{gaalpha}
 g( \gm, \alpha) = 
i_\infty\left( \bpm \on{sign}(\alpha)&\\&1 \epm
\right)
 i_{_{\text{finite}}}\left( \bpm \alpha &\\&1\epm \gamma^{-1}
\right).
 \end{equation}
  as in equation (36) 
  of \cite{GHL}. 

  \begin{lemma}
 Let $N \ge 1$ be an integer and let $f$ be a Maass form satisfying 
 $W_f(g) = \prod_v W_v(g_v)$ for all 
 $g = \{ g_v\} \in GL(2, \A).$  For $\gm =\bspm a&b\\c&d \espm\in SL(2, \Z)$ and $\alpha \in \Q^\times$ let 
 \begin{equation}\label{gaalpha}
 g( \gm, \alpha) = 
i_\infty\left( \bpm \on{sign}(\alpha)&\\&1 \epm
\right)
 i_{_{\text{finite}}}\left( \bpm \alpha &\\&1\epm \gamma^{-1}
\right).
 \end{equation}
 as in equation (36) 
  of \cite{GHL}.  Then $W_v( g( \gm, \alpha))$
 is a multiplicative function of $\alpha \in \Q^\times$ in all of the following cases:
 \begin{enumerate}
 \item $v=\infty$
 \item $v=p \nmid N,$
 \item $v=q\mid N$ and $v_q(c) = v_q(N),$
 \item $v=q\mid N$ and $v_q(d) = v_q(N),$
 \item $v=2 \mid N$ and $v_2(c) = v_2(N)-1.$
 \end{enumerate}
 \label{lem:5conditionsForMultiplicativeAtv}
 \end{lemma}
 \begin{proof}
 We have seen in the proof of theorem 35 
 of \cite{GHL} that  $W_v( g( \gm, \alpha))$ 
depends only on the sign of $\alpha$ in case (1) and only on $|\alpha|_p^{-1}$ in case (2).  In cases (3) and (5) it depends only on $|\alpha|_q^{-1}.$  Finally, in case (4) it is equal to $\widetilde \chi_{_{\text{idelic}}}\left( i_q\left( \bspm 1&0\\0&\alpha|\alpha|_q \espm\right)\right)$ times a function which depends only on $|\alpha|_q.$
 \end{proof}
 \begin{definition}\label{d:B_N_1}
 Let $N \ge 1$ be an integer and let $f$ 
 be a Maass form of level $N$ such that 
 $W_f(g) = \prod_v W_v(g_v)$ for all 
 $g = \{ g_v\} \in GL(2, \A).$ 
 Fix $\gm=\bspm a&b\\c&d \espm \in SL(2, \Z).$ 
 Set
 $N_1$
equal to $\prod q^{v_q(N) -v_q(c)},$ 
with the product taken over
  those primes $q\mid N$ which do not satisfy any of the five conditions listed in lemma 
\ref{lem:5conditionsForMultiplicativeAtv}.
Finally, set 
$B_{f,N_1}( \gm, \alpha)$ equal to the product
of the local Whittaker functions $W_q(g(\gm, \alpha)),$ taken over $q \mid N_1.$
 \end{definition}
 \begin{lemma}
 Keep the notation and assumptions of definition \ref{d:B_N_1}.  
 \label{l:bN1timesMultve}
 Then the function $B_f(\gm, \alpha)$ is 
 equal to $B_{f,N_1}(\gm, \alpha)$ times 
 a multiplicative function.
 \end{lemma}
 \begin{proof}
This is an obvious consequence of
lemma \ref{lem:5conditionsForMultiplicativeAtv}
and \eqref{e:baalpha=prod}. \end{proof}
 
 Of course, it does {\em not} follow from Lemma \ref{l:bN1timesMultve} that $B_f(\gm, \cdot)$ is 
 multiplicative if and only if $B_{f,N_1}(\gm, \cdot)$ is,
 as the discussion of the case $N=8$ at the end of section 6 
 in \cite{GHL} makes clear.  On the contrary, it is possible for $B_f(\gm, \cdot)$ to be multiplicative even when $B_{f,N_1}(\gm, \cdot)$ is not, provided $\prod_{v\nmid N_1} W_v( g( \gm , \alpha))$ vanishes sufficiently often.
 The next lemma and corollary deal with the question of when a product of two functions will be multiplicative.
 
 \begin{lemma}
Let $h: \Q^\times \to \C$ be a function. Then  $h$ is multiplicative if and only if 
$h( \alpha_1 \beta_1) h( \alpha_2\beta_2) = h( \alpha_1\beta_2)h(\alpha_2\beta_1)$ 
for all $\alpha_1, \alpha_2, \beta_1, \beta_2 \in \Q^\times$ such that 
$$|\alpha_1|_p\text{ or } |\alpha_2|_p \ne 1 \implies |\beta_1|_p = |\beta_2|_p =1 \qquad \forall \text{ primes }p.$$
\end{lemma}
\begin{proof}
The ``only if'' part is trivial.  To prove the ``if'' part, we may assume that $h$ is nonzero and fix $\alpha_0 \in \Q^\times$ with 
$h( \alpha_0) \ne 0.$  
Define 
$$\begin{aligned}
h_\infty( \text{sign}( \alpha_0)) &= h( \alpha_0),\\
h_\infty( - \text{sign}( \alpha_0)) &= h( -\alpha_0)/h( \alpha_0),\\
h_p(p^k) = h(\alpha_0 |\alpha_0|_p p^k)/h( \alpha_0).
\end{aligned}
$$
Then it is a straightforward check that
$h, h_\infty,$ and $\{ h_p: p \text{ prime}\}$
satisfy \eqref{e:multiplicativity}.
\end{proof}
\begin{remark}The condition $$|\alpha_1|_p\text{ or } |\alpha_2|_p \ne 1 \implies |\beta_1|_p = |\beta_2|_p =1 \qquad \forall \text{ primes }p$$
will be used repeatedly.
It can be thought of as a way of saying $\gcd(\alpha_i, \beta_j) =1, \; i,j=1,2,$ given that $\alpha_i,$ and $\beta_j$ are rational numbers.
\end{remark}

\begin{corollary}\label{c:multveCond}
If $h_1, h_2$ are two functions 
$\Q^\times\to \C$ and $h_1$ is multiplicative, then $h_1h_2$ is multiplicative if and only if 
$h_2( \alpha_1 \beta_1) h_2( \alpha_2\beta_2) = h_2( \alpha_1\beta_2)h_2(\alpha_2\beta_1)$ 
for all $\alpha_1, \alpha_2, \beta_1, \beta_2 \in \Q^\times$ such that 
$$|\alpha_1|_p\text{ or } |\alpha_2|_p \ne 1 \implies |\beta_1|_p = |\beta_2|_p =1 \qquad \forall \text{ primes }p,$$
{\em and} $h_1( \alpha_1\beta_1) h_1(\alpha_2\beta_2) \ne 0.$
\end{corollary}
In order to proceed further, we record an alternate formula for $W_q( g( \gm, \alpha))$ for values of $q$ not covered by lemma \ref{lem:5conditionsForMultiplicativeAtv}.
In order to place this result within a transparent conceptual setting, let us first point out that $B_f(\gm, \alpha)$ depends only on the bottom row of $\gm$ taken modulo $N,$ and that, up to scalar, it depends only on the corresponding point in $\P^1(\Z/N\Z).$ 
\begin{lemma}
Let $N \ge 1$ be an integer and $f$ a 
Maass form of level $N$ and character $\chi \pmod{N}.$ 
Take $\gm_1= \bspm a_1&b_1\\c_1&d_1 \espm, \gm_2= \bspm a_2&b_2\\c_2&d_2 \espm \in SL(2, \Z).$  If
$( c_1 , d_1) \equiv (c_2, d_2) \pmod{N},$
then $B_f(\gm_1, \alpha) = B_f( \gm_2, \alpha)$ for all $\alpha \in \Q^\times.$  If  
$[c_1:d_1]=[c_2: d_2]$ in $\P^1( \Z/N\Z),$ then there is a constant $c \in \C$ such that
$B_f(\gm_1, \alpha) = cB_f( \gm_2, \alpha)$
for all $\alpha \in \Q^\times.$
\end{lemma}
\begin{proof}
The points $[c_1:d_1],[c_2:d_2]\in \P^1( \Z/N\Z),$ are equal if and only if there exists
$\lambda \in (\Z/N\Z)^\times$ such that 
$(c_1, d_1) \equiv \lambda \cdot (c_2, d_2)\pmod{N}.$  
Supposing this to be the case, then 
$\gm_1 \gm_2^{-1}\equiv \bspm *&*\\0&\lambda \espm \pmod{N}.$  Hence $f|_k \gm_1 = \chi( \lambda) f|_k \gm_2,$ and the result follows, with the value of $c$ being $\chi(\lambda).$  If $( c_1 , d_1) \equiv (c_2, d_2) \pmod{N},$ then we are in the same situation, with $\lambda=1.$
\end{proof}
The next lemma has essentially the same proof.
\begin{lemma}
Let $N \ge 1$ be an integer and $f$ a 
Maass form of level $N$ and character $\chi \pmod{N}.$ 
Take $\gm_1= \bspm a_1&b_1\\c_1&d_1 \espm, \gm_2= \bspm a_2&b_2\\c_2&d_2 \espm \in SL(2, \Z).$
Define $g( \gm_i, \alpha),$ $(\alpha \in \Q^\times, i = 1,2)$ by \eqref{gaalpha}.  
Let $q \mid N$ be a prime.
If
$( c_1 , d_1) \equiv (c_2, d_2) \pmod{q^{v_q(N)}},$
then $W_q(g(\gm_1, \alpha)) = W_q(g(\gm_2, \alpha))$ for all $\alpha \in \Q^\times.$  If there exists 
$[c_1:d_1]=[c_2: d_2]$ in $\P^1( \Z/q^{v_q(N)}\Z),$ then there is a constant $c \in \C$ such that
$W_q(g(\gm_1, \alpha)) = cW_q(g( \gm_2, \alpha))$
for all $\alpha \in \Q^\times.$
\end{lemma}
In order to state the next proposition, it 
is convenient to introduce a few more definitions.
\begin{definition}
For $N\ge 1$ and integer, define the set 
$X_N = \left\{ (c, d) \in (\Z/N\Z)^2 \;|\; \langle c, d \rangle  = \Z/N\Z\right\}.$
Let $G_N$ be the subgroup of
$\Q^\times$ generated by  the primes which do not divide $N.$  
 Define a homomorphism $r_{N}: G_{N} \to (\Z/N\Z)^\times$ by mapping $m/n$ to $m\bar n$ for all $m, n \in \Z$ with $\gcd(m,n) =\gcd(m,N)=\gcd(n,N) = 1,$ where $\bar n$ denotes the inverse of $n$ in $(\Z/N\Z)^\times.$
For $\beta \in G_N$ and $x= (c,d) \in X_N$
set $\beta\cdot x= (r_N( \beta^{-1})c,d).$
If $\gamma =\bspm a&b\\ c&d \espm \in SL(2, \Z),$ define $x_N(\gamma) \in X_N$  to 
be $(c,d)$ taken modulo $N.$  If 
$f$ is a Maass form of level $N,$ 
define $B_f(x, \alpha)$ for $x\in X_N$ and $\alpha \in \Q^\times,$ to equal $B_f(\gm, \alpha)$ for any $\gm \in SL(2, \Z)$ with $x_N( \gm) = x.$  
\end{definition}
\begin{proposition}
 Let $N \ge 1$ be an integer and let $f$ 
 be a Maass form of level $N$ such that 
 $W_f(g) = \prod_v W_v(g_v)$ for all 
 $g = \{ g_v\} \in GL(2, \A).$ 
  For $\gm =\bspm a&b\\c&d \espm \in SL(2, \Z),$ and $\alpha \in \Q^\times$ 
 define $g(\gm, \alpha)$ by \eqref{gaalpha}.
 Define 
 $N_1$ and $B_{f,N_1}$ as in definition \ref{d:B_N_1}.  
 Fix $\alpha_0\in \Q^\times$ and $\gm'=\bspm a'&b'\\ c' & d' \espm  \in SL(2, \Z)$ such that 
 $$
 (c', d') \equiv (c,d),  
 \pmod{q^{v_q(N)}}, \qquad q \mid N_1,
 \qquad \qquad
 (c',d') \equiv (0,1), 
 \pmod{q^{v_q(N)}}, \qquad q\mid N, \; q\nmid N_1.
 $$
Set $x_0 = x_N(\gm').$
Then,
for $\beta\in G_{N_1},$ 
we have 
$B_{f,N_1}(\gm, \alpha_0  \beta ) 
=B_{f,N_1}( \beta \cdot x_0, \alpha_0)
= B_f(\beta \cdot x_0, \alpha_0).
$   
 \label{p:bN1(gm, ab)=B_f(b.gm,a)}
 \end{proposition}
\begin{proof}
Fix $\gm''_\beta$ such that $x_N( \gm''_\beta)= \beta \cdot x_0.$ 
To be specific take 
$$
\gm''_\beta = \bpm a''&b''\\c''&d'' \epm \equiv
\bpm 
a' & r_N(\beta) b'\\ r_N( \beta^{-1}) c'& d'  
\epm 
\pmod N.
$$
We first show that $B_{f,N_1}( \gamma, \alpha_0\beta) = B_{f,N_1}( \gm''_\beta, \alpha_0).$
From the definition of $B_{f,N_1},$ it
suffices to prove that $W_q(g( \gm, \alpha_0\beta) ) = W_q( g( \gm''_\beta, \alpha_0))$ for all $q \mid N_1.$   
 And this is clear, since 
it follows directly from the definitions that 
$$i_q\left( \bpm \alpha_0\beta\cdot |\alpha_0\beta|_q &0\\0&1 \epm \gm^{-1} \right)
= i_q\left( \bpm \alpha_0 |\alpha_0|_q &0\\0&1 
\epm \gm_\beta'' \right)\cdot i_q\left( \bpm \beta &\\&1\epm \right),$$
and it is clear that $i_q\left( \bspm \beta &\\&1\espm \right)$ lies in the kernel of $\widetilde \chi_{_{\text{idelic}}}$ in $K_0(N).$ 

To prove that $B_{f,N_1}( \gm_\beta'', \alpha_0) = B_f( \gm_\beta'', \alpha_0),$ 
we simply note that at all primes not dividing $N_1,$ the matrix 
$i_q\left( \bspm \alpha_0 |\alpha_0|_q &0\\0&1 
\espm \gm_\beta'' \right)$ is in the kernel of $\widetilde \chi_{_{\text{idelic}}}$ in $K_0(N),$ so that $W_q\left( i_q\left( \bspm \alpha_0 |\alpha_0|_q &0\\0&1 
\espm \gm_\beta'' \right)\right) =1.$
\end{proof}
\begin{corollary}\label{cor: BsubN1 factors through rN1}
Keep the notation from the statement of proposition 
\ref{p:bN1(gm, ab)=B_f(b.gm,a)}.  The function 
$\beta \mapsto B_{f,N_1}( \gm, \alpha\beta)$ 
factors through $r_{N_1}.$
\end{corollary}
\begin{proof}
It's clear from the definitions of $x_0,$ $N_1,$ 
and the action $\cdot$ that the function 
$\beta \mapsto \beta \cdot x_0$ factors through $r_{N_1},$ so this is immediate from proposition \ref{p:bN1(gm, ab)=B_f(b.gm,a)}. 
\end{proof}
The new language permits us to state
theorem 35 
of \cite{GHL}
in a simpler way.  This theorem
says 
 that $B_f(\gm, \alpha)$ is multiplicative whenever 
 the orbit of $x_N(\gm)$ under the action 
 of $G_N$ is a single point, or 
 maps to a single point in 
  $\P^1(\Z/N\Z).$ 
 We would like to prove something in the direction of a converse to this statement.   
However, the discussion of the case $N=8, \fa=\frac 12$ in section 6 of \cite{GHL} reveals that, in general, one must be prepared to work with orbits not of $(\Z/N\Z)^\times,$ but only of a subgroup corresponding to values of $n$
such that $B_f(I_2, n) \ne 0.$  (Cf. also corollary \ref{c:multveCond}.)

In order to prove an approximate
converse,
 we need to play two properties of $B_{f,N_1}(\gm, \cdot)$ off of one another: 
first, its restriction to $G_{N_1}$ factors through $r_{N_1},$ and second, if $B_f(\gm, \cdot)$
is multiplicative, the $B_{f,N_1}(\gm, \cdot)$ 
satisfies a weak multiplicativity property as in corollary \ref{c:multveCond}.
We first consider the simplest case, when
$B_{f,N_1}(\gm, \cdot)$ is actually multiplicative.  (This will be the 
case, e.g., if $B_f(I_2, p)$ is nonzero more than half of the time.)

\begin{lemma}\label{l:multve+factors=>DirChar}
Let $h: G_N \to \C$ be a nonzero function 
which factors through $r_N.$  Suppose further that 
$h( \alpha_1 \beta_1) h( \alpha_2\beta_2) = h( \alpha_1\beta_2)h(\alpha_2\beta_1)$ 
for all $\alpha_1, \alpha_2, \beta_1, \beta_2 \in G_N$ such that 
$$|\alpha_1|_p\text{ or } |\alpha_2|_p \ne 1 \implies |\beta_1|_p = |\beta_2|_p =1 \qquad \forall \text{ primes }p.$$
Then there is a Dirichlet character $\psi \pmod{N}$ 
and a complex number $c$ 
such that $h( \alpha) = c \psi(r_N(\alpha))$ for each $\alpha \in G_N.$
\end{lemma}
\begin{proof}
Fix $\alpha_0 \in G_N$ such that $h(\alpha_0) \ne 0,$ and for $m \in (\Z/N\Z)^\times,$ define $\psi(m) = h( \alpha \alpha_0)/h( \alpha_0),$ where $\alpha \in G_N$ is any element satisfying $r_N( \alpha ) = m.$  For any $m_1, m_2 \in (\Z/N\Z)^\times,$ we can choose 
primes $p_1, p_2$ not dividing $N,$ such 
that $v_{p_i}( \alpha_0) = 0,$ 
and $r_{N}(p_i)=m_i,\; (i=1,2).$
Hence
$$
\psi( m_1 m_2) = 
\frac{ h( \alpha_1\alpha_2 \alpha_0)}{h( \alpha_0)}
= \frac{ h( \alpha_1\alpha_2 \alpha_0)h( \alpha_0)}{h( \alpha_0)^2}
= \frac{ h( \alpha_1 \alpha_0)h(\alpha_2 \alpha_0)}{h( \alpha_0)^2}
= \psi( m_1) \psi(m_2).
$$
It follows that $\psi$ is a character of $(\Z/N\Z)^\times.$
\end{proof}

\section{Main results when $B_{f,N_1}(\gm, \cdot)$ is multiplicative}

\begin{theorem}\label{T:necCond(mu=0)-simple}
Let $N \ge 1$ be an integer and let 
 $f$ be a  Maass form of level $N$ and character $\chi \pmod{N}$ such that 
 $W_f(g) = \prod_v W_v(g_v)$ for all 
 $g = \{ g_v\} \in GL(2, \A).$ 
 Fix $\gm = \bspm a&b\\c&d \espm \in SL(2, \Z)$ and 
 define $N_1$ and $B_{f,N_1}(\gm, \cdot) $ as in definition \ref{d:B_N_1}. Assume that $B_{f,N_1}(\gm, \cdot)$ 
 is multiplicative,
 and that, for each residue class modulo $N_1$ there is an integer in that class such that $B_f(I_2,\, n ) \ne 0.$  Then the cusp parameter
 $\mu_{\fa, \chi}$ associated to the cusp $\fa= \gm\cdot \infty$ relative to the character $\chi$ is trivial.
\end{theorem}
\begin{proof}
If $\mu_{\fa, \chi} = \frac{r}{s} \ne 0,$ then 
$B_f( \gm , \alpha)$ is supported on rational numbers of the form, $\frac{r'}{sm_\fa}$ 
with $r' \equiv r \pmod{s}.$  
Now,  
$$
e^{2\pi i \mu_{\fa, \chi}} = \wt\chi\left( \gm\bpm 1&m_\fa\\ &1 \epm \gm^{-1} \right)  
$$
and $m_\fa$ is the least positive integer such that
$$
\gm \bpm 1&m_\fa\\ &1 \epm \gm^{-1} 
= \bpm 1-acm_\fa&a^2m_\fa \\ -c^2 m_\fa & 1+acm_\fa \epm \in \Gm_0(N).
$$
Hence $$m_\fa = \prod_{q\mid N}q^{\max(0, v_q(N)-2v_q(c)},$$
so that  $1+acm_\fa$ lies in 
\begin{equation} \label{set of k equiv 1 mod N2}\left\{ k\in \Z: k \equiv 1 \pmod{\prod_{q\mid N} q^{\max(v_q(c), v_q(N)-v_q(c))}}\right\}.\end{equation}
The image of \eqref{set of k equiv 1 mod N2} in $\Z/N\Z^\times$ has $N_3:= \prod_{q\mid N} q^{\min(v_q(c), v_q(N)-v_q(c))}$ elements.  
It follows that $s$ is a divisor of $N_3.$  In particular, every prime divisor of $s$ divides $N.$  

  Let $$
  B_{f,p}(\gm, \alpha) = W_{f,p}\left(\bpm\alpha & \\ & 1\epm\gamma^{-1}\right), \qquad 
  B_f^{N_1}( \gm, \alpha)=  B_f(I_2, \text{sign}(\alpha))
  \cdot \prod_p B_{f,p}( \gm, \alpha).$$ 
  Then $B_f( \gm , \alpha) = B_{f,N_1}( \gm, \alpha) B^{N_1}( \gm, \alpha).$ 
  Recall that for each prime $p \nmid N_1,$ 
$B_{f,p}(\gm,\frac{r'}{sm_\fa})$ depends only 
on $v_p(\frac{r'}{sm_\fa}),$ and that for $p\nmid N$
it is independent of $\gm.$

  Assuming $B_{f,N_1}( \gm, \alpha)$ is multiplicative,  it follows from
  corollary \ref{cor: BsubN1 factors through rN1}
and 
  lemma \ref{l:multve+factors=>DirChar}
that $$B_{f,N_1}(\gm, \frac{r'}{sm_\fa}) 
= \psi( r') B_{f,N_1}( \gm, \frac{1}{sm_\fa})
\qquad ( \forall,r' \in \Z, \; \gcd( r', N_1)=1).$$
Hence $B^{N_1}( \gm, \frac{r'}{sm_\fa})=0$
for all $r' \in \Z$ with $\gcd(r', N_1) =1$
and $r'\not\equiv r \pmod{s}.$
But this is inconsistent with our assumption regarding $B_f(I_2, n).$ 
\end{proof}

\begin{theorem}\label{t:necCond2-simple}
Let $N \ge 1$ be an integer and let 
 $f$ be a  Maass form of level $N$ and character $\chi \pmod{N}$ such that 
 $W_f(g) = \prod_v W_v(g_v)$ for all 
 $g = \{ g_v\} \in GL(2, \A).$ 
 Fix $\gm = \bspm a&b\\c&d \espm \in SL(2, \Z)$ and 
 define  $B_{f,N_1}(\gm, \alpha), \;(\alpha \in \Q^\times)$ as in definition \ref{d:B_N_1}. 
 Set $M=\gcd(N, 2cd).$  
 Assume that $B_{f,N_1}(\gm, \cdot)$ 
 is multiplicative.
 Then $N/M$ is a divisor of $24.$  In 
 particular, $N_1$ is not divisible by any 
 prime greater than $3.$  
\end{theorem}
\begin{proof}
Define $x_0 \in X_N$ from $\gm$ as in proposition \ref{p:bN1(gm, ab)=B_f(b.gm,a)}.
Then, by proposition  \ref{p:bN1(gm, ab)=B_f(b.gm,a)}, we have 
$$B_{f,N_1}(\gm, \alpha \beta) = B_{f,N_1}( x_0, \alpha \beta) = B_{f,N_1}( \beta\cdot x_0, \alpha)\qquad \forall \alpha \in \Q^\times, \;\beta \in G_{N_1},$$ and by 
lemma \ref{l:multve+factors=>DirChar}
we have 
$$
B_{f,N_1}(\gm, \alpha\beta) = \psi( r_{N_1}( \beta)) B_{f,N_1}( \gm , \alpha)
\qquad \forall \alpha \in \Q^\times, \;\beta \in G_{N_1}.
$$ 
It follows that
$$B_{f,N_1}( \beta \cdot x_0, \alpha) = \psi(r_{N_1}( \beta)) B_{f,N_1}( \gm, \alpha).$$
Comparing contributions from other places, we find that
$$B_f(\beta \cdot x_N(\gamma), \alpha) = \psi( r_{N_1}( \beta)) B_f( \gm, \alpha), 
\qquad \forall \alpha \in \Q^\times, \beta \in G_{N}.
$$
Hence for any $\beta \in G_N, \delta_\beta \in SL(2, \Z)$ with $x_N( \delta_ \beta ) = \beta \cdot x_N( \gm),$ we have
$f|_k \delta_\beta = \psi(r_{N_1}(\beta))
f|_k \gm,$
and so 
$$f|_k \delta_\beta\gm^{-1} = \psi( r_{N_1}(\beta)) f.$$
Now let $\Gm_f \subset SL(2, \Z)$ be the subgroup consisting of all matrices $\gm_1 \in SL(2, \Z)$ such that $f|_k \gm_1 = \xi( \gm_1) f$
for some $\xi(\gm_1) \in \C^\times.$  Clearly, $\Gm_0(N) \subset \Gm_f,$
and so is 
\begin{equation}\label{e:setOfdgmInverse}
\left\{ \left.\bpm a' &b'\\ c'&d' \epm \gm^{-1} \;\right|\; 
 d' \equiv d \pmod{N},\; c' \equiv mc \pmod{N}, \text{ some }m \in (\Z/N\Z)^\times\right\}.\end{equation}
 \begin{proposition}\label{Gm0N+dbgmInv=Gm0M}
 Let $N \ge 1$ be an integer, and fix 
 $\gm = \bspm a&b\\ c&d \espm \in SL(2, \Z).$  Set $M=\gcd(N, 2cd).$  The subgroup 
 of $SL(2, \Z)$ generated by $\Gm_0(N)$ and
 \eqref{e:setOfdgmInverse} is $\Gm_0(M).$
 \end{proposition}
 \begin{proof}
 Let $\Gm$ denote the group in question.  
 Clearly, $\Gm$ contains the principal congruence subgroup $\Gm(N),$ so that it suffices to compute the image of $\Gm$ in 
 $SL(2, \Z/N\Z),$ which we denote $\bar \Gm.$
   Now, 
 $$d' \equiv d \pmod{N},\; c' \equiv mc \pmod{N}
 \implies 
 \bpm a' &b'\\ c'&d' \epm \gm^{-1}
 \equiv\bpm  * & * \\ cd(m-1) & * \epm 
 \pmod{N}.$$
 This proves that $\Gm \subset \Gm_0(M).$  
 It is permissible to vary $\gm$ by an element of $\Gm_0(N),$ and hence to place additional conditions on $a$ and $b.$  Therefore, we may assume that 
 $$
 q\mid \gcd(N,d) \implies a\equiv 0 \pmod{q^{v_q(N)}}, \qquad 
 q\mid N, \, q\nmid d \implies b \equiv 0 \pmod{q^{v_q(N)}}.
 $$
 For $m \in (\Z/N\Z)^\times,$ set 
 $$
 h(m) = \bpm a& b\bar m\\ cm & d \epm 
 \gm^{-1} = \bpm 1- (\bar m -1) bc & ab(\bar m -1) \\ cd( m-1) & 1-bc( m-1)\epm
 \equiv 
 \begin{cases}
 \bpm 1&0\\cd(m-1) & 1 \epm, & q \nmid d,\\ \bpm \bar m&0\\cd(m-1) & m\epm, & q \mid d,
 \end{cases}
 \pmod{q^{v_q(N)}}.
 $$
 (Here $\bar m$ is the inverse of $m$ in $\Z/N\Z.$)
 Clearly, $h(m)\in \bar \Gm,$  for each $m.$
 
 Assume $2 \nmid \gcd(N,d).$  
 We show that $\bar \Gm$ contains the subgroup generated by $\bspm 1&0\\ M & 1\espm.$ Using the isomorphism
 $SL(2, \Z/N\Z) \cong \prod_{q\mid N} SL(2, \Z/q^{v_q(N)}\Z),$ it suffices to prove that for each $q \mid N,$ $\bar \Gm$ contains  
 an element congruent to $\bspm 1&0\\ \epsilon q^{v_q(M)} & 1 \espm \pmod{q^{v_q(N)}},$
 for some $\epsilon \in \Z/N\Z^\times,$ and congruent to the identity modulo $q_1^{v_{q_1}(N)}$ for each $q_1 \mid N, q_1 \ne q.$  If $q \nmid d,$ use $h(m),$
 with
 $$m \equiv -1 \pmod{q^{v_q(N)}}, \qquad
 m \equiv 1 \pmod{q_1^{v_{q_1}(N)}}, \; q_1\mid N, q_1\ne q.$$
 If $q\mid d,$ use $h(m)^2,$ for $m$ subject to the same conditions.
 
 If $2\mid \gcd(N,d),$ the same approach shows that the subgroup of $SL(2, \Z/N\Z)$ generated by $\bspm 1&0\\2M& 1\espm$ is
 contained in $\bar \Gm.$ 
 
  Proposition \ref{Gm0N+dbgmInv=Gm0M}
 is an immediate consequence of lemma 
 \ref{B+M=abcMd} when $2\nmid \gcd(N,d).$ If $2 \mid \gcd(N,d)$ then lemma  \ref{B+M=abcMd} implies that $\Gm_0(2M) \subset \Gm.$  Now, for any even $M,$ 
 and any $\gm_1,\gm_2 \in \Gm_0(M)\smallsetminus \Gm_0(2M),$ the matrix
 $\gm_1^{-1} \gm _2$ is in $\Gm_0(2M).$  
 It follows that 
 $\Gm_0(M)$ is generated by $\Gm_0(2M)$ and any single element of $\Gm_0(M) \smallsetminus  \Gm_0(2M).$ 
 Any matrix in $SL(2, \Z)$ which maps to 
$h(-1)\in SL(2, \Z/N\Z)$ lies in $\Gm_0(M) \smallsetminus  \Gm_0(2M),$ so proposition  \ref{Gm0N+dbgmInv=Gm0M} follows in the case $2\mid \gcd(d,N)$ as well.
  \end{proof}
  It follows directly from proposition \ref{Gm0N+dbgmInv=Gm0M} that the group $\Gm_f$ of all matrices $\gm_1 \in SL(2, \Z)$ such that $f|_k \gm_1 = \xi( \gm_1) f$
for some $\xi(\gm_1) \in \C^\times$
contains $\Gm_0(N).$   
  Further the function $\xi: \Gm_f \to \C$ is
a character, whose restriction to $\Gm_0(N)$ is $\widetilde \chi.$  

Now, $f$ is of level $N.$  It follows that $N$ is the smallest integer satisfying 
\begin{enumerate}
\item $\Gm_0(N) \in \Gm_f,$ and
\item $\xi(\gm_1) = 1, \; \forall \, \gm_1 \in \Gm_1(N).$ 
\end{enumerate} 
It is possible that $M<N,$ but only if the 
restriction of $\xi$ to $\Gm_1(M)$ is nontrivial.  However, if  $M_1$ is given in terms of $M$ as 
in proposition \ref{Gm1(M1)}, then 
 $\Gm_0(M_1) \in \Gm_f,$ and
 $\xi(\gm_1) = 1, \; \forall \, \gm_1 \in \Gm_1(M_1).$ 
Therefore, 
by the definition of level, $M_1 \ge N.$
This completes the proof of theorem 
\ref{t:necCond2-simple}. 
\end{proof}
In fact, we have proved the following sharper result.
\begin{theorem}\label{t:necCond2-simple-sharp}
Keep the notation from the statement of theorem \ref{t:necCond2-simple}.  For each prime $q,$ we have
$$
0 \le v_q\left( \frac NM\right) \le 
\begin{cases}
0,& q>3, \text{ and/or } v_q(M) =0,\\
1, & q=3, v_q(M) >0,\\
2, & q=2, v_q(M) =1,\\
3, & q=2, v_q(M) \ge 2.
\end{cases}
$$
\end{theorem}

\section{Main results in general}

We would now like to treat the general case.
That is, instead of assuming the function $B_{f,N_1}(\gm, \cdot)$ is multiplicative, we show that if 
$B_{f}(\gm, \cdot)$ is multiplicative, then $B_{f,N_1}(\gm, \cdot)$ satisfies a ``weak multiplicativity'' property, with the weakness determined by how often $B_f^{N_1}( \gm, \cdot)$ vanishes.  We then deduce analogues of 
theorems \ref{T:necCond(mu=0)-simple} and \ref{t:necCond2-simple} appropriate to this setting.
 
First, we require an analogue of lemma \ref{l:multve+factors=>DirChar}.
\begin{lemma}
Let $N\ge 1$ be an integer, and fix a multiplicatively closed subset 
$S \subset  G_N.$  Let $\bar S$ denote the image of $S$ in $(\Z/N\Z)^\times.$ Take $h: \Q^\times \to \C$ a nonzero function satisfying 
\begin{equation}\label{e:weakMultiplicativityRelativeToASetS}\begin{aligned}
h( \alpha_1 \beta_1) h( \alpha_2\beta_2)& = h( \alpha_1\beta_2)h(\alpha_2\beta_1) \\
&\forall \; \alpha_1, \alpha_2, \beta_1, \beta_2 \in S\text{ s.t. } |\alpha_1|_p\text{ or } |\alpha_2|_p \ne 1 \implies |\beta_1|_p = |\beta_2|_p =1 \qquad \forall \text{ primes }p.\end{aligned}\end{equation}
Assume that $h$ factors through 
$(\Z/N\Z)^\times.$  If $G \subset (\Z/N\Z)^\times$ is a subgroup, satisfying 
\begin{equation}\label{e:PropertyImposedOnASubgroupOfZNZToGetACharacter}
\forall g_1\in G, g_2 \in \bar S, \;\exists \beta_1,\beta_2 \in S, \; \text{ s.t. }   |\beta_1|_p \ne 1 \Rightarrow  |\beta_2|_p = 1, \text{ all primes }p, \text{ and }
r_N(\beta_1) = g_1, \; r_N(\beta_2) = g_2,
\end{equation}
then there is a character $\psi$ of $G$ such that 
$
h( \alpha\beta) = \psi( r_N(\alpha))h(\beta),
$ for all $\alpha \in G_N$ with $r_N(\alpha) \in G,$ and all $\beta \in G_N$ such 
that $r_N(\beta) \in \bar S.$\label{weakMultve+factors=>CharOnCosets}
\end{lemma}
\begin{proof}
The same arguments used to prove lemma \ref{l:multve+factors=>DirChar} show that
$h(\alpha)$ is a constant multiple of $\psi(r_N(\alpha))$ for some character $\psi$ of $G$ and all $\alpha$ with $r_N(\alpha) \in G.$  The extension is clear from \eqref{e:PropertyImposedOnASubgroupOfZNZToGetACharacter} and the weak multiplicativity property \eqref{e:weakMultiplicativityRelativeToASetS}.
\end{proof}
We would like to apply lemma 
\ref{weakMultve+factors=>CharOnCosets}
to the analysis of the function 
$B_{f,N_1}$ introduced in definition
\ref{d:B_N_1}.  It follows from Corollary \ref{c:multveCond} that $B_{f,N_1}$ satisfies \eqref{e:weakMultiplicativityRelativeToASetS}
relative to the set $S$ of all $\alpha \in \Q^\times$
such that $
 B_f^{N_1}(\gm, n)\ne 0.$
 What is not clear, however, is that the image 
 $\bar S$ of this set in $(\Z/N_1\Z)^\times$ will contain any subgroups at all.
 (For example, $\bar S$ might not contain $1.$)
 We address this point in the next section.  For now, we 
 assume that some group $G$ as in lemma 
\ref{weakMultve+factors=>CharOnCosets}
exists, and deduce analogues of
theorems \ref{T:necCond(mu=0)-simple} and 
\ref{t:necCond2-simple}.
\begin{corollary}
Let $N\ge 1$ be an integer and let
$f$ be a Maass form of level $N$ 
such that $W_f(g) = \prod_v W_v(g_v)$ for all 
 $g = \{ g_v\} \in GL(2, \A).$ 
 Fix $\gm = \bspm a&b\\ c& d \espm \in SL(2, \Z)$ and define $N_1$ and $B_{f,N_1}$ 
 as in \ref{d:B_N_1}.  Also, define 
\begin{equation}\label{e:B_f^N,B_N^N1}
B_f^{N}(\gm , \alpha) 
:=W_\infty\left( \bspm \text{\rm sign}(\alpha)&\\&1 \espm \right)
 \prod_{p\nmid N}W_p\left( \bspm \alpha&\\&1 \espm \gm^{-1}\right), 
 \qquad
 B_N^{N_1}(\gm , \alpha) 
:=
 \prod_{q\mid N, \, q\nmid N_1}W_q\left( \bspm \alpha&\\&1 \espm \gm^{-1} \right).
\end{equation}
Recall that 
for $q\mid N, \, q\nmid N_1,$ 
the function $W_q\left( \bspm \alpha&\\&1 \espm \gm^{-1} \right)$ depends 
only on $v_q(\alpha).$
Fix $\alpha_0$ in the subgroup of 
$\Q^\times$ generated by the primes dividing $N,$ and such that 
$W_q\left( \bspm q^{v_q(\alpha_0)}&\\&1 \espm \gm^{-1} \right)\ne 0,$
for each $q\mid N, \, q\nmid N_1.$   

Recall further that $B_f^N( \gm , \alpha)=0$ 
whenever $\alpha \notin \Z.$  
Let $S= \{ n \in \Z\mid B_f^N(\gm, n) \ne 0\}.$
Consider 
the function $B_{f,N_1}(\gm, \alpha_0 \cdot n),\; n \in \Z\cap G_N.$  If $G \subset (\Z/N\Z)^\times$ is a subgroup 
satisfying \eqref{e:PropertyImposedOnASubgroupOfZNZToGetACharacter}, 
then there is a character $\psi$ such that 
$$B_{f,N_1}(\gm, \alpha_0 \cdot n)
= B_{f,N_1}( n\cdot x_{N_1}(\gm), \alpha_0)
= \psi( r_N(n)) B_{f,N_1}( \gm, \alpha),$$
and 
\begin{equation}
\label{e:B_f(n.gm,alpha0)=psi(n)B_f(gm,alpha_0)}
B_f(\gm, n\cdot \alpha_0) =
\psi(r_N(n)) B_f(\gm, \alpha_0)= B_f(n\cdot x_N(\gm), \alpha_0),\end{equation}
for all $n \in S, r_N(n) \in G.$ 
\label{c:B_f(gm, a.n)=B_f(n.gm,a)=psi(n)B_f(gm,a)}
\end{corollary}
\begin{proof}
The fact that $B_{f,N_1}(\gm, \alpha_0 \cdot n)
= B_{f,N_1}( n\cdot x_{N_1}(\gm), \alpha_0)$
was given in proposition \ref{p:bN1(gm, ab)=B_f(b.gm,a)}.  Equality with $\psi( r_N(n)) B_{f,N_1}( \gm, \alpha)$ follows from 
 lemma \ref{weakMultve+factors=>CharOnCosets}.
 The second identity follows from comparing the contributions at the other primes.
\end{proof}

\begin{theorem}\label{T:mu=0-general}
Let $N\ge 1$ be an integer and let
$f$ be a Maass form of level $N$ 
such that $W_f(g) = \prod_v W_v(g_v)$ for all 
 $g = \{ g_v\} \in GL(2, \A).$ 
 Fix $\gm = \bspm a&b\\ c& d \espm \in SL(2, \Z).$
  Let $S= \{ n \in \Z\mid B_f^N(\gm, n) \ne 0\}.$
If there is a group $G$ satisfying condition \eqref{e:PropertyImposedOnASubgroupOfZNZToGetACharacter} relative to $S,$ such that the orbit of $x_N(\gm)$ under the action of $G$ has more than one element then the cusp parameter of $\fa = \gm \cdot \infty$ is trivial.
\end{theorem}
\begin{proof}
Same as theorem \ref{T:necCond(mu=0)-simple}.
\end{proof}

\begin{theorem}\label{Main theorem--precise version at theend}
Let $N$ be an integer,
and let $f$ be a Maass form of level 
$N$ such that 
$W_f(g) = \prod_v W_v(g_v)$ for all 
 $g = \{ g_v\} \in GL(2, \A).$
 Fix $\gm = \bspm a&b\\c&d \espm \in SL(2, \Z).$  Define $B_f^N( \gm, \alpha), \, \alpha \in \Q^\times$ as in \eqref{e:B_f^N,B_N^N1}.
 Let $S= \{ n \in \Z\mid B_f^N(\gm, n) \ne 0\},$
and let $G \subset (\Z/N\Z)^\times$ be a subgroup which satisfies \eqref{e:PropertyImposedOnASubgroupOfZNZToGetACharacter} relative to $S.$  
Let $M$ be the greatest divisor of $N$ which also divides $(m-1)cd$ for every $m \in G.$  Define $M_1$ from $M$ as in 
proposition \ref{Gm1(M1)}.  
Suppose that the Fourier coefficients $B_f(\gm, \alpha), \; \alpha \in \Q^\times$
are multiplicative.  Then $M_1$ is a multiple of $N.$ \end{theorem}
\begin{proof}
The proof is the same as that of 
theorem \ref{t:necCond2-simple}. 
We let $\Gm_f = \{ \gm_1 \in SL(2, \Z)\mid f|_k\gm_1 = \xi(\gm_1)\cdot f\}$
and observe that $\Gm_f$ contains the subgroup 
 $\Gm \subset SL(2, \Z)$ generated by $\Gm_0(N)$ and 
$$
\left\{ \left.
\bpm a'&b'\\c' & d' \epm\gm^{-1} \in SL(2, \Z)\;\,
\right| \;\,
d' \equiv d \pmod{N}, \;
c' \equiv m c \pmod{N}, \text{ for some } m \in G
\right\}.
$$
It is easily checked as before that $\Gm$ is contained in $\Gm_0(M)$ for $M$ as above.  To check the $\Gm \supset \Gm_0(M)$ it suffices to check that 
$\bspm 1& 0 \\ r M & 1 \espm \in \Gm$ for some $r \in \Z$ with $\gcd(r, N) =1.$  

As before we may assume that 
$$
q \mid \gcd(N,d) \implies q^{v_q(N)} \mid a,
\qquad q \mid N, \, q\nmid d \implies q^{v_q(N)} \mid b.
$$
Then for each $m \in (\Z/N\Z)^\times$ there exists $\gm_m \in \Gm$ with 
$$
\gm_m
 \equiv 
 \begin{cases}
 \bpm 1&0\\cd(m-1) & 1 \epm, & q \nmid d,\\ \bpm \bar m&0\\cd(m-1) & m\epm, & q \nmid d,
 \end{cases}
 \pmod{q^{v_q(N)}}.
$$
The ``diagonal part'' is equivalent to an element of $\Gm_0(N), \pmod{N},$ 
and so there is an element of $\Gm$ 
which is equivalent to 
$$
\begin{cases}
 \bpm 1&0\\cd(m-1) & 1 \epm, & q \nmid d,\\ \bpm 1&0\\mcd(m-1) & 1\epm, & q \nmid d,
 \end{cases}
 \pmod{q^{v_q(N)}}.
$$ 
The existence of $\bspm 1& 0 \\ r M & 1 \espm \in \Gm$  with $\gcd(r, N) =1$
follows.    
As before, it follows as before that 
$\xi\big|_{\Gm_1(M_1)}$ is trivial, 
which is a contradiction unless $N\mid M_1,$
because $N$ is the level of $f.$
\end{proof}

\section{A lower bound for the size of the group $G$ occuring in theorem \ref{Main theorem--precise version at theend}}

Theorem \ref{T:mu=0-general}
begs the question of whether any group $G\subset (\Z/N\Z)^\times$ 
satisfying \eqref{e:PropertyImposedOnASubgroupOfZNZToGetACharacter}
will exist.  
The answer is yes.  In fact the subgroup 
of squares will satisfy a condition somewhat 
stronger than \eqref{e:PropertyImposedOnASubgroupOfZNZToGetACharacter}.

\begin{lemma}
Let $N\ge 1$ be an integer and let 
 $f$ be a Maass form of level $N$
 which satisfies $W_f(g) = \prod_v W_v(g_v)$ for all 
 $g = \{ g_v\} \in GL(2, \A).$  
For any square $m \in (\Z/N\Z)^\times,$ there exist infinitely many pairwise relatively prime integers $n$ 
  such that $B_f^N(I_2, n)\ne 0$ and $r_N(p^k) =m.$
  (Here, $I_2$ denotes the $2\times 2$ identity matrix.)
\end{lemma}
\begin{proof}
 First, the requirement that 
 $W_f(g) = \prod_v W_v(g_v)$ for all 
 $g = \{ g_v\} \in GL(2, \A)$
 implies that 
 the cuspidal automorphic representation $\pi$ 
 of $GL(2, \A)$ generated by 
 $f_{_{\text{adelic}}}$
 is irreducible.  (This is a consequence of
 the strong multiplicity one theorem,
lemma 3.1 of \cite{Langlands:1980} )
  By the tensor product theorem \cite{Flath:1979}, we have $\pi \cong \bigotimes_v' \pi_v,$
 with $\pi_\infty$ a $(\frak g, K_\infty)$-module and $\pi_p$ a representation of $GL(2, \Q_p).$   
 Further, if $p \nmid N,$ the local 
 representation
 $\pi_p$ is unramified and $W_p$ is the 
 unique $GL(2, \Z_p)$ fixed vector of the 
 Whittaker model of $\pi_p,$ which is normalized so that its value on $GL(2, \Z_p)$ is one.  
  By the Casselman-Shalika formula 
 \cite{C-S}, 
 we have  
 $$W_p\left( \bpm n&\\&1 \epm \right)
 = W_p\left( \bpm p^{v_p(n)}&\\&1 \epm \right) = \frac{\alpha_p^{v_p(n)+1} - \beta_p^{v_p(n)+1}}{\alpha_p-\beta_p},$$
 for some $\alpha_p, \beta_p \in \C, \, \alpha_p \ne \beta_p.$  
 It follows that if 
$\left\{ n\in \Z, n >0 \mid W_p\left( \bspm n&\\&1 \espm \right) =0\right\}$ is nonempty, then it is equal to $\{n \in \Z, n>0\mid v_p(n) \equiv -1 \pmod{m_p} \}$ for some integer $m_p >1.$  

Recall that any finite abelian group is isomorphic to a product of cyclic groups of prime power order.  First suppose that $m\in (\Z/N\Z)^\times$ is of odd order.  Write $\ord(m)$ for the order of $m,$ and let $p$ be a prime with $r_N(p)=m.$  Then $B_f^N( \infty, p)=0$ only if $m_p=2,$ in which case $B_f^N( \infty, p^{k\cdot \ord(m)+1})$ 
for any odd $k.$  Further, at least one of 
$B_f^N( \infty, p),\;B_f^N( \infty, p^{\ord(m)+1})$ is nonzero for any prime with $r_N(p)=m,$ of which there are infinitely many.

Now suppose that $m$ lies in a cyclic subgroup of $(\Z/N\Z)^\times$ of order $2^\ell,$ and that $m$ is a square.  Take $u$ with $\ord(u) =2^r, \; r \le \ell$ and $u^2=m.$  
Then for any prime $p$ such that $r_N(p) =u,$ either $B_f^N( I_2, p^2) \ne 0,$ or else $m_p = 3,$ in which case $B_f^N(I_2, p^{\ord(u)+2})$ is nonzero. 
Since every square in $(\Z/N\Z)^\times$ is the product of an element of odd order and 
a square in a cyclic group whose order is a power of two, this completes the proof.
\end{proof}

This yields an unconditional version of theorem \ref{Main theorem--precise version at theend}.  
\begin{theorem}\label{suffcond2--final}
Let $N$ be an integer,
and let $f$ be a Maass form of level 
$N$ such that 
$W_f(g) = \prod_v W_v(g_v)$ for all 
 $g = \{ g_v\} \in GL(2, \A).$
 Fix $\gm = \bspm a&b\\c&d \espm \in SL(2, \Z).$
If $B_f( \gm, \cdot)$ is multiplicative, 
then $N \mid 576 cd.$
\end{theorem}

\begin{remarks}\label{remarks at the end }
If $A_f(\infty, p) \ne 0$ for more than half of all 
primes $p \nmid N,$ then the group 
$(\Z/N\Z)^\times$ itself satisfies 
\eqref{e:PropertyImposedOnASubgroupOfZNZToGetACharacter}, and we are in the situation of theorem \ref{t:necCond2-simple}.  
According to Lemma 4.1 
of \cite{KLSW}, this condition is satisfied by  holomorphic forms of weight $\ge 2$  with trivial character which are not of CM type.
The same argument gives the same result any Maass form $f$ such that the second, fourth, and sixth symmetric power $L$ functions are all holomorphic. 
It is expected that this condition holds
unless the representation $\pi$ attached to $f$ satisfies one of the following
\begin{itemize}
\item  $\pi \cong \pi \otimes \eta$ for some character $\eta: \Q^\times \bs \A^\times \to \C^\times,$
\item $\pi$ contains the adelic lift of a weight one Maass form corresponding to a holomorphic form,
\item $\pi$ contains the adelic lift of a weight zero Maass form with Laplace eigenvalue $1/4.$  
\end{itemize}
(See \cite{Shahidi:1994} \S 7, \cite{Gelbart:1997},\S4.2.) 

On the other hand, if the first of the three conditions above {\em is} satisfied, then it is possible for $A_f(\infty ,p)$ to vanish at least half of the time.  For example, by a theorem of Gauss, the elliptic curve $y^2 = x^3 + 4x$ has exactly $p+1$ points on it whenever $p$ is congruent to $3 \pmod{4}.$ This implies that the associated cusp form of weight two satisfies $A_f(\infty,p) = 0$ for every prime congruent to $3 \pmod{4}.$
The level of this cusp form is $32$
(see \cite{cremona}, p.103). 
\end{remarks}


\end{document}